\documentclass[12pt,reqno]{amsart}
\usepackage{amscd,amssymb,amsthm}
\usepackage{graphicx}

\usepackage{enumerate}
\theoremstyle{plain}
\newtheorem{theorem}{Theorem}[section]
\newtheorem{corollary}[theorem]{Corollary}
\newtheorem{lemma}[theorem]{Lemma}

\newcommand{\floor}[1]{\left\lfloor{#1}\right\rfloor}
\newcommand{\reel}{\mathbb{R}}
\newcommand{\comp}{\mathbb{C}}

\newcommand{\vf}{\varphi}
\newcommand{\abs}[1]{\left\vert #1\right\vert }
\newcommand{\adh}[1]{\overline{#1} }

\newcommand{\bg}{\medskip\goodbreak}
\newcommand{\itemref}[1]{\eqref{#1}}
\newenvironment{enumeratea}{\begin{enumerate}%
	[\upshape (a)]}{\end{enumerate}}

\title[Partial Fraction Expansions for Square Root Approximants]
{Partial Fraction Expansions for Newton's and Halley's Iterations for Square Roots}
\author[Omran Kouba]{Omran Kouba$^\dag$}
\address{Department of Mathematics \\
Higher Institute for Applied Sciences and Technology\\
P.O. Box 31983, Damascus, Syria.}
\email{omran\_kouba@hiast.edu.sy}
\thanks{$^\dag$ Department of Mathematics, Higher Institute for Applied Sciences and Technology.}

\begin{document}
\parindent=0pt
\begin{abstract}
When Newton's method, or Halley's method is used to approximate the $p${th} root of $1-z$, a sequence 
of rational functions is obtained. In this paper, a beautiful formula for these rational functions is proved in the square root case, using an interesting link to Chebyshev's polynomials. It allows the  determination of the sign of the coefficients of the power series expansion of these rational functions. This answers positively the square root case of a proposed conjecture by Guo(2010). \par
\end{abstract}
\smallskip\goodbreak

\parindent=0pt
\maketitle

\section{\bf Introduction }\label{sec1}
\bg
\parindent=0pt

\qquad There are several articles and research papers that deal with
 the determination of the sign pattern of the coefficients of a power series expansions (see \cite{kra},\cite{str} and the bibliography therein). In this work we consider this problem
in a particular case.\bg
\qquad Let $p$ be an integer greater than $1$, and $z$ any complex number. If we apply Newton's method to solve the equation $x^p=1-z$ starting from the initial value $1$, we obtain
the sequence of rational functions $(F_k)_{k\geq0}$ in the variable $z$ defined by the following iteration

\begin{equation*}
F_{k+1}(z)=\frac{1}{p}\left((p-1)F_k(z)+\frac{1-z}{F_k^{p-1}(z)}\right),\qquad F_0(z)\equiv 1.
\end{equation*}

\qquad Similarly, if we apply Halley's method to solve the equation $x^p=1-z$ starting from the same initial value $1$, we get
the sequence of rational functions $(G_k)_{k\geq0}$ in the variable $z$ defined by the iteration

\begin{equation*}
G_{k+1}(z)=\frac{(p-1)G_k^p(z)+(p+1)(1-z)}{(p+1)G_k^{p}(z)+(p-1)(1-z)}G_k(z),\qquad G_0(z)\equiv 1.
\end{equation*}

\qquad It was shown in \cite{guo} and more explicitly stated in \cite{kou} that both $F_k$ and 
$G_k$ have  power series expansions that are convergent in the neighbourhood of $z=0$, and that
these expansions start similar to the power series expansion of $z\mapsto\root{p}\of{1-z}$. More precisely,
the first $2^k$ coefficients of the power series expansion of $F_k$ are identical to the corresponding
coefficients in the power series expansion of $z\mapsto\root{p}\of{1-z}$, and
the same holds for the first $3^k$  coefficients of the power series expansion of $G_k$. It was conjectured
\cite[Conjecture 12]{guo}, that the coefficients of the power series expansions of $(F_k)_{k\geq2}$, $ (G_k)_{k\geq1}$ and  $z\mapsto\root{p}\of{1-z}$ have the same sign pattern. 
\bg
\qquad In this article, we will consider only the case $p=2$. In this case an unsuspected link to Chebyshev polynomials of the first and the second kind is discovered, it will allow us to find general formul\ae~ 
for $F_k$ and $G_k$ as sums of partial fractions. This will allow us to prove Guo's conjecture in this particular case. Finally, we note that for $p\geq3$ the conjecture remains open. 
\bg
\qquad Before proceeding to our results, let us fix some notation and definitions. We will use freely 
the properties of Chebyshev polynomials $(T_n)_{n\geq0}$ and
$(U_n)_{n\geq0}$ of the first and the second kind. In particular, they can be defined for $x>1$ by the folmul\ae   :
\begin{equation}\label{E:eq1}
\forall\,\vf\in\reel,\quad T_n(\cos\vf)=\cos(n\vf),\quad\hbox{and}\quad U_n=\frac{1}{n+1}T'_n.
\end{equation}

\qquad For these definitions and more on the properties of these polynomials we invite the reader
to consult \cite{mas} or any general treatise on special functions for example \cite[Chapter 22]{abr}, \cite[\$13.3-4]{arf} 
or \cite[Part six]{pol}, and the references therein. \bg
\qquad Let $\Omega =\comp\setminus[1,+\infty)$, (this is the complex plane cut along the real numbers greater or equal to 1.) For $z\in \Omega$, we denote by $\sqrt{1-z}$ the square root of $1-z$ with positive real part. We know that $z\mapsto \sqrt{1-z}$
is holomorphic in $\Omega$. Moreover, 
\begin{equation}\label{E:bin}
\forall\, z\in\adh{D(0,1)},\qquad \sqrt{1-z}=\sum_{m=0}^\infty\lambda_m z^m,
\end{equation}
where
\begin{equation}\label{E:bin1}
 \forall\,m\geq0,\qquad\lambda_m=\frac{1}{(1-2m)2^{2m}}\binom{2m}{m}.
\end{equation}

\qquad The standard result, is that \eqref{E:bin} is true for $z$ in the open unit disk, but the fact that for $n\geq1$ we have $\lambda_n=\mu_{n-1}-\mu_n<0$, where $\mu_n=2^{-2n}\binom{2n}{n}\sim1/\sqrt{\pi n}$, proves the uniform convergence of the series $\sum\lambda_mz^m$ in
the closed unit disk, and \eqref{E:bin} follows by Abel's Theorem since $z\mapsto \sqrt{1-z}$ can be continuously extended to  $\adh{D(0,1)}$.
\bg
\qquad Finally, we consider the sequences of
rational functions $(V_n)_{n\geq0}$, $(F_n)_{n\geq0}$ and $(G_n)_{n\geq0}$ defined by
\begin{align}
V_{n+1}(z)&=\frac{1-z+V_n(z)}{1+V_n(z)}, && V_0(z)\equiv 1,\label{E:eqv}\\
F_{n+1}(z)&=\frac{1}{2}\left(F_n(z)+\frac{1-z}{F_n (z)}\right), && F_0(z)\equiv 1,\label{E:eqf}\\
G_{n+1}(z)&=\frac{ G_n^3(z)+3(1-z)G_n(z)}{3G_n^{2}(z)+1-z}, && G_0(z)\equiv 1,\label{E:eqg}
\end{align}
\bg
where $(F_n)_{n\geq0}$ and $(G_n)_{n\geq0}$ are Newton's and Halley's iterations mentioned before, (in the case $p=2$.) Since the sequence $(V_n)_{n\geq0}$ is simpler than the other two sequences, we 
will prove our main result for the $V_n$'s, then  we will deduce the corresponding properties for $F_n$ and $G_n$.
\bigskip\goodbreak

\section{\bf The Main Results }\label{sec2}
\bg

\qquad We start this section by proving a simple property of that shows why it is sufficient to
study the sequence of $V_n$'s to deduce the properties of Newton's and Halley's
iterations the $F_n$'s and $G_n$'s :
\bg
\begin{lemma}\label{lm1}
The sequences $(V_n)_n$, $(F_n)_n$ and $(G_n)_n$ of the rational functions defined inductively by \eqref{E:eqv},  \eqref{E:eqf} and  \eqref{E:eqg}, satisfy the following properties :
\begin{enumeratea}
\item For  $z\in\Omega$ and $n\geq0$, we have :
\[
\frac{V_n(z)-\sqrt{1-z}}{V_n(z)+\sqrt{1-z}}
  =\left(\frac{1-\sqrt{1-z}}{1+\sqrt{1-z}}\right)^{n+1},
\]\label{itlm11}
\item For $n\geq0$ we have $F_n=V_{2^n-1}$, and $G_n=V_{3^n-1}$.\label{itlm12}
\end{enumeratea}
\end{lemma}
\begin{proof}
First, let us suppose that $z=x\in(0,1)$. In this case, we see by induction
that all the terms of the sequence $(V_n(x))_{n\geq0}$ are well defined and positive,
and we have the following recurrence relation :
 \begin{align*}
\frac{V_{n+1}(x)-\sqrt{1-x}}{V_{n+1}(x)+\sqrt{1-x}}&=
\frac{1-x+V_n(x)-\sqrt{1-x}(1+V_n(x))}{1-x+V_n(x)+\sqrt{1-x}(1+V_n(x))}\\
&=\frac{1-\sqrt{1-x}}{1+\sqrt{1-x}}\cdot
\frac{V_n(x)-\sqrt{1-x}}{V_n(x)+\sqrt{1-x}}.
\end{align*}
Therefore, using simple induction, we obtain
\begin{equation}\label{E:eq2}
\forall\,n\geq0,\qquad
  \frac{V_n(x)-\sqrt{1-x}}{V_n(x)+\sqrt{1-x}}
  =\left(\frac{1-\sqrt{1-x}}{1+\sqrt{1-x}}\right)^{n+1},
\end{equation}
Now, for a given $n\geq0$, let us define $R(z)$ and $L(z)$ by the formul\ae\ :
\[
L(z)=\frac{V_n(z)-\sqrt{1-z}}{V_n(z)+\sqrt{1-z}}\quad\hbox{and}\quad R(z)= \left(\frac{1-\sqrt{1-z}}{1+\sqrt{1-z}}\right)^{n+1}.
\]
Clearly, $R$ is analytic in $\Omega$ since $\Re(\sqrt{1-z})>0$ for $z\in\Omega$. On the other hand, if $V_n=A_n/B_n$ where $A_n$ and $B_n$ are two co-prime polynomials then 
\[
L(z)=\frac{(A_n(z)-\sqrt{1-z}B_n(z))^2}{A_n^2(z)-(1-z)B_n^2(z)}.
\]
So $L$ is meromorphic with, (at the most,) a finite number
of poles in $\Omega$. Using analyticity, we conclude from \eqref{E:eq2} that
$L(z)=R(z)$ for $z\in\Omega\setminus\mathcal{P}$ for some finite set $\mathcal{P}
\subset\Omega$, but this implies that the points in $\mathcal{P}$ are
removable singularities, and that $L$ is holomorphic and identical to $R$ in
$\Omega$. This concludes the proof of \itemref{itlm11}.
\bg
\qquad To prove \itemref{itlm12}, consider $z=x\in(0,1)$, as before,
we see by induction
that all the terms of the sequences $(F_n(x))_{n\geq0}$ and $(G_n(x))_{n\geq0}$ 
are well-defined and positive. Also, we check easily, using
\eqref{E:eqf} and \eqref{E:eqg}, that the following recurrence relations 
hold :
\begin{align*}
\frac{F_{n+1}(x)-\sqrt{1-x}}{F_{n+1}(x)+\sqrt{1-x}}&=
\left(\frac{1-\sqrt{1-x}}{1+\sqrt{1-x}}\right)^2\cdot
\frac{F_{n}(x)-\sqrt{1-x}}{F_{n}(x)+\sqrt{1-x}},\\
\noalign{\hbox{and}}\\
\frac{G_{n+1}(x)-\sqrt{1-x}}{G_{n+1}(x)+\sqrt{1-x}}&=
\left(\frac{1-\sqrt{1-x}}{1+\sqrt{1-x}}\right)^3\cdot
\frac{G_{n}(x)-\sqrt{1-x}}{G_{n}(x)+\sqrt{1-x}}.
\end{align*}
It follows that
\begin{align*}
\frac{F_{n}(x)-\sqrt{1-x}}{F_{n}(x)+\sqrt{1-x}}&=
\left(\frac{1-\sqrt{1-x}}{1+\sqrt{1-x}}\right)^{2^n}=
\frac{V_{2^n-1}(x)-\sqrt{1-x}}{V_{2^n-1}(x)+\sqrt{1-x}},\\
\noalign{\hbox{and}}\\
\frac{G_{n}(x)-\sqrt{1-x}}{G_{n}(x)+\sqrt{1-x}}&=
\left(\frac{1-\sqrt{1-x}}{1+\sqrt{1-x}}\right)^{3^n}=
\frac{V_{3^n-1}(x)-\sqrt{1-x}}{V_{3^n-1}(x)+\sqrt{1-x}}.
\end{align*}
Hence $F_n(x)=V_{2^n-1}(x)$ and $G_n(x)=V_{3^n-1}(x)$ for every
$x\in(0,1)$, and \itemref{itlm12} follows by analyticity.
\end{proof}

\qquad Lemma \ref{lm1} provides a simple proof  of the following known result.\bg

\begin{corollary}\label{co1}
The sequences 
$(V_n)_n$, $(F_n)_n$ and $(G_n)_n$ of the rational functions defined inductively by \eqref{E:eqv},  \eqref{E:eqf} and  \eqref{E:eqg}, converge 
uniformly on compact subsets of $\Omega$ to the function $z\mapsto\sqrt{1-z}$.
\end{corollary}
\begin{proof}
Indeed, this follows from Lemma \ref{lm1} and the fact that for every non-empty compact set $K\subset\Omega$ we have :
\[
\sup_{z\in K}\abs{\frac{1-\sqrt{1-z}}{1+\sqrt{1-z}}}<1,
\]
which is easy to prove.
\end{proof}
\bg
\qquad One can also use Lemma \ref{lm1} to study $V_n$ in the neighbourhood of $z=0$ :
\bg
\begin{corollary}\label{co2}
For every $n\geq0$, and for $z$ in the neighbourhood of $0$ we have
\begin{align}
V_n(z)&=\sqrt{1-z}+O(z^{n+1}),\label{E:eq3}\\
\noalign{\hbox{and}}
V_n(z)&=\sum_{m=0}^n\lambda_mz^m+O(z^{n+1}).\label{E:eq4}
\end{align}
where $\lambda_m$ is defined by \eqref{E:bin1}
\end{corollary}

\begin{proof}
Indeed, using Lemma \ref{lm1}  we have :
\begin{align*}
V_n(z)&=\sqrt{1-z}\cdot
\frac{(1+\sqrt{1-z})^{n+1}+(1-\sqrt{1-z})^{n+1}}{(1+\sqrt{1-z})^{n+1}-(1-\sqrt{1-z})^{n+1}},\\
&=\sqrt{1-z}+\frac{2\sqrt{1-z}}{(1+\sqrt{1-z})^{n+1}-(1-\sqrt{1-z})^{n+1}}\cdot (1-\sqrt{1-z})^{n+1},\\
&=\sqrt{1-z}+\frac{2\sqrt{1-z}}{(1+\sqrt{1-z})^{n+1}-(1-\sqrt{1-z})^{n+1}}\cdot \frac{z^{n+1}}{(1+\sqrt{1-z})^{n+1}},\\
&=\sqrt{1-z}+\frac{2\sqrt{1-z}}{(1+\sqrt{1-z})^{2(n+1)}-z^{n+1}}\cdot z^{n+1},
\end{align*}
In particular, for $z$ in the neighbourhood of $0$, we have 
$V_n(z)=\sqrt{1-z}+O(z^{n+1})$, which is \itemref{E:eq3}. 

On the other hand, using \eqref{E:bin}, 
we obtain 
\[
V_n(z)=\sum_{m=0}^n\lambda_mz^m+O(z^{n+1})
\]
which is \itemref{E:eq4}.
\end{proof}

\qquad Recall that we are interested in the sign pattern of the coefficients
of the power series expansion of $F_n$ and $G_n$, in the neighbourhood of $z=0$.
Lemma \ref{lm1} reduces the problem to finding sign pattern of the coefficients
of the power series expansion of $V_n$. But, $V_n$ is rational function
and a partial fraction decomposition would be helpful. The next theorem 
is our main result :
\bg
 
\begin{theorem}\label{th21}
Let $n$ be a positive integer, and let $V_n$ be the rational function defined by the recursion \eqref{E:eqv}.  Then the partial fraction decomposition of $V_n$ is as follows :
\begin{align*}
V_1(z)&=1-\frac{z}{2},\\
V_n(z)&=1-\frac{z}{2}-\frac{z^2}{2(n+1)}\sum_{k=1}^{\floor{n/2}}\frac{
\sin^2(\frac{2\pi k}{n+1})}{1-z\cos^2(\frac{\pi k}{n+1})},&&\hbox{for $n\geq2$.}
\end{align*}
\end{theorem}
\bg
\begin{proof}
Let us recall the fact that for $x>1$  Chebyshev polynomials of the first and the second kind satisfy the following identities :
\begin{align}
T_n(x) & =\frac{(x+\sqrt{x^2-1})^n+(x-\sqrt{x^2-1})^n}{2},\label{E:eqch1}\\
U_n(x) & =\frac{(x+\sqrt{x^2-1})^{n+1}-(x+\sqrt{x^2-1})^{n+1}}{2\sqrt{x^2-1}}.\label{E:eqch2}
\end{align}
Thus, for $0<x<1$ we have
\begin{align*}
x^nT_n\left(\frac1x\right) & =\frac{(1+\sqrt{1-x^2})^n+(1-\sqrt{1-x^2})^n}{2},\\
x^n U_n\left(\frac1x\right) & =\frac{(1+\sqrt{1-x^2})^{n+1}-(1+\sqrt{1-x^2})^{n+1}}{2\sqrt{1-x^2}}.
\end{align*}
So, using by Lemma \ref{lm1}(a),  we find that
\begin{equation}\label{E:eq5}
V_n(x^2)=\frac{x^{n+1}T_{n+1}(1/x)}{x^nU_{n}(1/x)}=
\frac{P_n(x)}{Q_n(x)},
\end{equation}
where $P_n(x)=x^{n+1}T_{n+1}(1/x)$ and $Q_n(x)=x^{n}U_{n}(1/x)$.
\bg
It is known that $U_n$ has $n$ simple zeros, namely $\{\cos(k\theta_n):1\leq k\leq n\}$ where $\theta_n=\pi/(n+1)$. So, if we define

\begin{align*}
\Delta_{2m}&=\{1,2,\ldots,2m\}\\
\noalign{\hbox{ and}}\\
\Delta_{2m-1}&=\{1,2,\ldots,2m-1\}\setminus\{m\},
\end{align*}
then the singular points of the rational function ${P_n}/{Q_n}$ are $\{\lambda_{k}:k\in\Delta_n\}$
with $\lambda_{k}=\sec(k\theta_n)$. Moreover, since
\begin{align*}
P_n(\lambda_k)&=\lambda_k^{n+1}T_{n+1}\left(\cos(k\theta_n)\right)\\
&=\lambda_k^{n+1}\cos(\pi k)=(-1)^k\lambda_k^{n+1}\ne0,
\end{align*}

we conclude that these singular points are, in fact, simple poles with residues given by
\begin{equation*}
\hbox{Res}\left(\frac{P_n}{Q_n},\lambda_k\right)
=-\lambda_k^3\frac{ T_{n+1}(1/\lambda_k)}{U_n^\prime(1/\lambda_k)}.
\end{equation*}
But, from the identity $U_n(\cos\vf)=\sin((n+1)\vf)/\sin\vf$ we conclude that
\begin{equation*}
U_n^\prime(\cos\vf)=\frac{\cos\vf \sin((n+1)\vf)}{\sin^3\vf}
-(n+1)\frac{\cos((n+1)\vf)}{\sin^2\vf},
\end{equation*}
hence, 
\begin{equation*}
U_n^\prime\left(\cos(k\theta_n)\right)=(n+1)\frac{(-1)^{k+1}}{\sin^2(k\theta_n)},
\end{equation*}
and finally, 
\begin{equation}\label{E:eq6}
\hbox{Res}\left(\frac{P_n}{Q_n},\lambda_k\right)
=\frac{\lambda_k^3}{n+1}\sin^2(k\theta_n )
=\frac{\lambda_k}{n+1}\tan^2(k\theta_n ).
\end{equation}
From this we conclude that the rational function $R_n$ defined by
\begin{align}
R_n(x)&=\frac{P_n(x)}{Q_n(x)}-\frac{1}{n+1}\sum_{k\in\Delta_n}\frac{\lambda_k\tan^2(k\theta_n )}{x-\lambda_k}\notag\\
&=\frac{P_n(x)}{Q_n(x)}+\frac{1}{n+1}\sum_{k\in\Delta_n}\frac{\tan^2(k\theta_n )}{1-x\cos(k\theta_n )},
\label{E:eq7}
\end{align}

is, in fact, a polynomial, and $\deg R_n=\deg P_n-\deg Q_n\leq (n+1)-(n-1)=2$. \bg

Noting that $k\mapsto n+1-k$ is a permutation of $\Delta_n$ we conclude that
\begin{align*}
\sum_{k\in\Delta_n}\frac{\tan^2(k\theta_n )}{1-x\cos(k\theta_n )}&=
\sum_{k\in\Delta_n}\frac{\tan^2(k\theta_n )}{1+x\cos(k\theta_n )}\\
&=\frac{1}{2}\sum_{k\in\Delta_n}\left(\frac{\tan^2(k\theta_n )}{1+x\cos(k\theta_n )}+\frac{\tan^2(k\theta_n )}{1-x\cos(k\theta_n )}\right)\\
&=\sum_{k\in\Delta_n}\frac{\tan^2(k\theta_n )}{1-x^2\cos^2(k\theta_n )}
\end{align*}
and using the fact that
\begin{equation*}
\frac{1}{1-z \cos^2\vf}=1+z\cos^2\vf+\frac{z^2\cos^4\vf}{1-z \cos^2\vf},
\end{equation*}
we find
\begin{align*}
\sum_{k\in\Delta_n}\frac{\tan^2(k\theta_n )}{1-x\cos(k\theta_n )}
&=\sum_{k\in\Delta_n}\tan^2(k\theta_n )
+x^2\sum_{k\in\Delta_n}\sin^2(k\theta_n )\\
&\qquad+
x^4 \sum_{k=1}^n\frac{\cos^2(k\theta_n ) \sin^2(k\theta_n )}{1-x^2\cos^2(k\theta_n )}
\end{align*}
where we added a ``zero'' term to the last sum for odd $n$. 

Combining this conclusion with \eqref{E:eq7}
we conclude that there exists a polynomial $S_n$ with $\deg S_n\leq 2$ such that
\begin{equation*}
S_n(x)=\frac{P_n(x)}{Q_n(x)}+\frac{x^4}{4(n+1)} \sum_{k=1}^n\frac{\sin^2(2k\theta_n )}{1-x^2\cos^2(k\theta_n )}.
\end{equation*}
Moreover, $S_n$ is even, since both  $P_n$ and $Q_n$ are even. Thus, going back to \eqref{E:eq5} we conclude  that there exists two constants $\alpha_n$ and $\beta_n$ such that
\begin{equation}\label{E:eq8}
V_n(z)=\alpha_n+\beta_n z
+\frac{z^2}{4(n+1)} \sum_{k=1}^n\frac{ \sin^2(2k\theta_n )}{1-z\cos^2(k\theta_n )}
\end{equation}
But, from Corollary \ref{co2}, we also have $V_n(z)=1-\frac{1}{2}z+O(z^2)$ for $n\geq 1$, thus
$\alpha_n=1$ and $\beta_n=-1/2$. 

\qquad Finally, noting that the terms corresponding to $k$ and $n+1-k$ in the sum \eqref{E:eq8}
are identical, we arrive to the desired formula.
This concludes the proof of Theorem \ref{th21}.
\end{proof}

\bg
\qquad Theorem \ref{th21} allows us to obtain a precise information about the power series
 expansion of $V_n$ in the neighbourhood of $z=0$ :\bg

\begin{corollary}\label{co22}
Let $n$ be a positive integer greater than 1, and let $V_n$ be the rational function defined by \eqref{E:eqv}.  
Then the radius of convergence of power series expansion $\sum_{m=0}^\infty A_{m}^{(n)}z^m$ of $V_n$ in the neighbourhood of $z=0$ is $\sec^2(\frac{\pi}{n+1})>1$, and the coefficients $(A_{m}^{(n)})_{m\geq0}$ satisfy the following properties :
\begin{enumeratea}
\item For $0\leq m\leq n$ we have $A_{m}^{(n)}=\lambda_m$, where $\lambda_m$ is defined by
\eqref{E:bin1}.\label{coit221}
\item For $m>n$ we have $A_{m}^{(n)}<0$ and 
\[
\sum_{m=n+1}^\infty\left(-A_{m}^{(n)}\right)=\frac{1}{2^{2n}}\binom{2n}{n}-\frac{1}{n+1}.
\]\label{coit222}
\end{enumeratea}
Moreover, for every $n\geq0$ and every $z$ in the closed unit disk $\adh{D(0,1)}$ we have
\[
\abs{V_n(z)-\sqrt{1-z}}\leq \frac{2}{\sqrt{\pi n}}\abs{z}^{n+1}.
\]
 In particular, $(V_n)_{n\geq0}$ converges uniformly on  $\adh{D(0,1)}$ to the function $z\mapsto\sqrt{1-z}$.
\end{corollary} 
\bg
\begin{proof}
Let us denote $\pi/(n+1)$ by $\theta_n$. By Theorem \ref{th21} the poles of $V_n$, for $n>1$,  are $\{\sec^2(k\theta_n):1\leq k\leq\floor{n/2}\}$ and the nearest one to $0$ is $\sec^2(\theta_n)$. This proves the statement about the radius of convergence.
\bg
\qquad Also, we have seen in Corollary \ref{co2},  that in the neighbourhood of $z=0$
we have
\[
V_n(z)=\sum_{m=0}^n\lambda_mz^m+O(z^{n+1}).
\]
This proves that for $0\leq m\leq n$ we have $A_m^{(n)}=\lambda_m$ which is \itemref{coit221}.

\qquad To prove \itemref{coit222}, we note that for $1\leq k\leq \floor{n/2}$, and  $z\in D(0,\sec^2(\theta_n)$,  we have
\[
\frac{1}{1-z\cos^2(k\theta_n)}=\sum_{m=0}^\infty\cos^{2m}(k\theta_n)z^m,
\]
hence
\[
V_n(z)=1-\frac{z}{2}-\sum_{m=0}^\infty\left(\frac{1}{2(n+1)}\sum_{k=1}^{\floor{n/2}}\cos^{2m}(k\theta_n)
\sin^2(2k\theta_n)\right)\,z^{m+2}.
\]
This gives the following alternative formul\ae~  for $A_{m}^{(n)}$ when $m\geq2$ :
\begin{align}
A_{m}^{(n)}&=-\frac{2}{n+1}\sum_{k=1}^{\floor{n/2}}\cos^{2(m-1)}(k\theta_n)
\sin^2(k\theta_n),\notag\\
&=-\frac{1}{n+1}\sum_{k=1}^{n}\cos^{2(m-1)}(k\theta_n)
\sin^2(k\theta_n),\label{E:eq9}
\end{align}
where \eqref{E:eq9} is also valid for $m=1$. This proves in particular that $A_{m}^{(n)}<0$ for $m>n$.
\bg
\qquad Since the radius of convergence of $\sum_{m=0}^\infty A_m^{(n)}z^m$ is greater than $1$
we conclude that
\begin{equation}\label{E:eq10}
\sum_{m=0}^\infty A_m^{(n)}=V_n(1),
\end{equation}
but, we can prove by induction from \eqref{E:eqv} that $V_n(1)=1/(n+1)$. Therefore, \eqref{E:eq10} implies that, for $n\geq 1$ we have
\begin{align*}
\sum_{m=n+1}^\infty (-A_m^{(n)})&=-\frac{1}{n+1}+\sum_{m=0}^n A_m^{(n)}=-\frac{1}{n+1}+1-\sum_{m=1}^n \lambda_m,\\
&=-\frac{1}{n+1}+1-\sum_{m=1}^n \left(\mu_{m-1}-\mu_m\right),\\
&=\mu_n-\frac{1}{n+1},
\end{align*}
where $\mu_m=2^{-2m}\binom{2m}{m}$, which is \itemref{coit222}.

\qquad For $n\geq 1$ and $z\in\adh{D(0,1)}$ we have
\begin{align}
\abs{V_n(z)-\sum_{m=0}^n\lambda_m z^m}&\leq
\abs{\sum_{m=n+1}^\infty A_{m}^{(n)} z^m}\notag\\
&\leq\abs{z}^{n+1}\cdot\sum_{m=n+1}^\infty (-A_{m}^{(n)})\notag\\
&\leq\left(\frac{1}{2^{2n}}\binom{2n}{n}-\frac{1}{n+1}\right)\abs{z}^{n+1}.\label{E:eq11}
\end{align}
On the other hand, using \eqref{E:bin} we see that for $n\geq 1$ and $z\in\adh{D(0,1)}$ we have
 \begin{align}
\abs{\sqrt{1-z}-\sum_{m=0}^n\lambda_m z^m}&\leq
\abs{\sum_{m=n+1}^\infty \lambda_m z^m}\notag\\
&\leq\abs{z}^{n+1}\cdot \sum_{m=n+1}^\infty (\mu_{m-1}-\mu_{m})\notag\\
&\leq\frac{1}{2^{2n}}\binom{2n}{n}\abs{z}^{n+1}.\label{E:eq12}
\end{align}
Combining \eqref{E:eq11} and \eqref{E:eq12}, and noting that $2^{-2n}\binom{2n}{n}\leq 1/\sqrt{\pi n}$ we obtain
\[
\abs{V_n(z)-\sqrt{1-z}}\leq \frac{2}{\sqrt{\pi n}}\abs{z}^{n+1},
\]
which is the desired conclusion.
\end{proof}
\bg
\qquad The following corollary is an immediate consequence of Lemma \ref{lm1} and Corollary \ref{co22}. It proves
that Conjecture 12 in \cite{guo} is correct in the case of square roots. 
\bg
\begin{corollary}\label{co3} The following properties of $(F_n(z))_{n\geq0}$  and
$(G_n(z))_{n\geq0}$, the
Newton's and Halley's approximants  of $\sqrt{1-z}$, defined by the recurrences \eqref{E:eqf} and \eqref{E:eqg} hold :

\begin{enumeratea}
\item For $n>1$, the rational function $F_n(z)$ has a power series expansion $1-\sum_{m=1}^\infty B_{m}^{(n)}z^m$ with $\sec^2(2^{-n}\pi)$ as
radius of convergence, and $B_{m}^{(n)}>0$ for every $m\geq1$. Moreover,
\[
\forall\,z\in\adh{D(0,1)},\qquad\abs{F_n(z)-\sqrt{1-z}}\leq\frac{2}{\sqrt{\pi}}\cdot\frac{\abs{z}^{2^n}}{\sqrt{2^{n}-1}}.
\]
\item For $n\geq1$, the rational function $G_n(z)$ has a power series expansion $1-\sum_{m=1}^\infty C_{m}^{(n)}z^m$ with $\sec^2(3^{-n}\pi)$ as radius of convergence, and $C_{m}^{(n)}>0$ for every $m\geq1$. Moreover,
\[
\forall\,z\in\adh{D(0,1)},\qquad\abs{G_n(z)-\sqrt{1-z}}\leq\frac{2}{\sqrt{\pi}}\cdot\frac{\abs{z}^{3^n}}{\sqrt{3^{n}-1}}.
\]
\end{enumeratea}
\end{corollary}
\bg
\textit{Remark.} Guo's conjecture \cite[Conjecture 12]{guo} in the case of square roots is just the fact that $B_m^{(n)}>0$ for $n>1$, $m\geq1$, and that
$C_m^{(n)}>0$ for $n,~m\geq1$. The case of $p$th roots for $p\geq3$ remains open.
\bg
\textbf{Conclusion.} In this paper, we presented a link between the rational functions approximating $z\mapsto\sqrt{1-z}$
obtained from the application of Newton's and Halley's method, and Chebyshev Polynomials. This was used
to find the partial fraction decomposition of this rational functions, and the sign pattern of the coefficients of their power series
expansions was obtained. Finally, this was used to prove the square root case of Conjecture 12 in \cite{guo}, and to give estimates of the approximation error for $z$ in the closed unit disk.


\end{document}